\documentclass[11pt,reqno]{amsart}
\usepackage{extarrows}
\usepackage{rotating}
\usepackage{bbm}
\usepackage{mathrsfs}
\usepackage{stmaryrd}
\usepackage{amscd,latexsym,amsthm,amsfonts,amssymb,amsmath,amsxtra}

\usepackage[hyperfootnotes=false, colorlinks, citecolor=RoyalBlue, linkcolor=Bittersweet, urlcolor=link]{hyperref}
\usepackage{color}
\usepackage[usenames,dvipsnames]{xcolor}
\usepackage{mathpazo}
\usepackage[mathcal]{eucal}
\usepackage{graphicx}
\usepackage[all]{xy}

\usepackage{setspace}
\usepackage[new]{old-arrows}

\usepackage{tikz}
\usetikzlibrary{calc, intersections}
\usepackage{tikz-cd}

\newcommand{\BC}{{\mathbb {C}}}

\newcommand{\BW}{{\mathbb {W}}} \newcommand{\BX}{{\mathbb {X}}}
\newcommand{\BY}{{\mathbb {Y}}}

\newcommand{\CA}{{\mathcal {A}}} \newcommand{\CB}{{\mathcal {B}}}
 
 \newcommand{\CF}{{\mathcal {F}}}

\newcommand{\CS}{{\mathcal {S}}}

\newcommand{\diag}{{\mathrm{diag}}}\renewcommand{\d}{{\mathrm{d}}}

\newcommand{\GL}{{\mathrm{GL}}}
\newcommand{\Hom}{{\mathrm{Hom}}}

\newcommand{\id}{{\mathrm{id}}}
\newcommand{\Ind}{{\mathrm{Ind}}} \newcommand{\ind}{{\mathrm{ind}}}
\newcommand{\uind}{{\mathrm{uind}}} \newcommand{\uInd}{{\mathrm{uInd}}}
\newcommand{\Irr}{{\mathrm{Irr}}}

\newcommand{\rank}{{\mathrm{rank}}}

\newcommand{\Rep}{{\mathrm{Rep}}}

\newcommand{\Stab}{{\mathrm{Stab}}}

\newcommand{\Tr}{{\mathrm{Tr}}}

\newcommand{\wt}{\widetilde}

\newcommand{\pair}[1]{\langle {#1} \rangle}

\newcommand{\incl}{\hookrightarrow}

\newcommand{\sk}{\medskip}
\newcommand{\lra}{\longrightarrow}
\newcommand{\ra}{\rightarrow}

\newcommand{\s}{\sk\noindent}

\theoremstyle{plain}
\newtheorem{thm}{Theorem}[section] 
\newtheorem{lem}[thm]{Lemma}  \newtheorem{prop}[thm]{Proposition}

\theoremstyle{definition}
 
\newtheorem{rem}[thm]{Remark}

\numberwithin{equation}{section}

\title{Local theta correspondence and generalized Shalika models}
\author{Chong Zhang}

\begin{document}
\date{}
\maketitle

\begin{abstract}
We study the behavior of generalized Shalika models under the local theta correspondence for type II pairs of equal rank. This generalizes a result of Gan on the relation between Shalika models and linear models.
\end{abstract}

\tableofcontents

\section{Introduction}

\subsection{Overview} Degenerate and generalized Whittaker models for representations of reductive groups over local fields are important in the theory of automorphic representations. They have numerous applications to local character expansions, Fourier coefficients, and L-functions of  automorphic forms.
Gomez and Zhu \cite{gz} studied the behavior of the generalized Whittaker models under the local theta correspondence for reductive dual pairs of type I. In this paper, we study the same problem for type II dual pairs, restricting ourselves to the special case of the so-called generalized Shalika models. A new feature is that we encounter models attached to subgroups with nontrivial reductive parts, unlike the models appearing in \cite{gz}.  
 
Now we introduce some relevant notions more precisely.
Let $F$ be a non-archimedean local field of characteristic 0. Let $n\geq 1$ and $k\geq2$ be two positive integers. Denote by $\GL_{kn}$ the general linear group of rank $kn$ over $F$. We also use 
$\GL_{kn}$ to denote the group $\GL_{kn}(F)$ of $F$-points.  Then $(\GL_{kn},\GL_{kn})$ is a reductive dual pair of type II. Denote by $\Irr(\GL_{kn})$ the set of equivalence classes of irreducible smooth representations of $\GL_{kn}$. For $\pi\in\Irr(\GL_{kn})$, denote by $\Theta(\pi)$ and $\theta(\pi)$ the big and small theta lifts of $\pi$, respectively. It is known that $\theta(\pi)$ is isomorphic to the contragredient 
$\pi^\vee$ of $\pi$. 

Consider the unipotent subgroup
$$N:=\left\{x=\begin{pmatrix}1_n&X_1&*&*&*\\ &1_n&X_2&*&*\\ &&\ddots&\ddots&*\\ &&&1_n&X_{k-1}\\ &&&&1_n \end{pmatrix}\right\}$$ of $\GL_{kn}$, that is, $N$ is the unipotent radical of the standard parabolic subgroup of $\GL_{kn}$ attached to the partition $(n,\cdots,n)$ of $kn$. Fix a nontrivial additive character $\psi$ of $F$. We also denote by $\psi$ the character of $N$ given by
$$\psi(x):=\prod_{i=1}^{k-1}\psi(\Tr(X_i)),$$ where $\Tr$ is the trace map. The stabilizer (under the conjugation action) of the pair $(N,\psi)$ in $\GL_{kn}$  is the subgroup
$$\GL_n^{\Delta}:=\left\{\diag(g,\cdots,g)\mid g\in\GL_n\right\}.$$

For $\pi\in\Irr(\GL_{kn})$ and $\sigma\in\Irr(\GL_n)$, the space
$$\Hom_{\GL_n^{\Delta}\cdot N}\left(\pi,\sigma\otimes\psi\right)$$ is called the space of {\em generalized Shalika models}. If $k=2$ (resp. $k=3$) and $\sigma$ is the trivial representation (or, more generally, a character), these are known as the Shalika models (resp. Ginzburg--Rallis models) in the literature. Forgetting the 
$\GL_n^{\Delta}$-action,  the space
$$\Hom_{N}\left(\pi,\psi\right)$$ is a special case of degenerate Whittaker models, called the $(k,n)$ models in \cite{cfgk} (or simply the Shalika models in \cite{lm}). When $k=2$, Gan \cite[Theorem 3.1]{gan19} showed that
\begin{equation}\label{equ gan}\Hom_{\GL_n^{\Delta}\cdot N}\left(\Theta(\pi),\sigma\otimes\psi\right)\cong\Hom_{\GL_n\times\GL_n}(\pi^\vee,\sigma\otimes\BC),\end{equation} where $\GL_n\times\GL_n=\left\{\diag(g_1,g_2)\mid g_1,g_2\in\GL_n\right\}$. The latter space in \eqref{equ gan} is the space of (Friedberg--Jacquet) linear models.
The goal of this paper is to generalize the isomorphism \eqref{equ gan}  to arbitrary $k$.

\subsection{Main results}
Let us introduce another subgroup 
$$H:=\left\{x=\begin{pmatrix}1_n&X_1&*&*&*&*\\ &1_n&X_2&*&*&*\\ &&\ddots&\ddots&*&*\\ &&&1_n&X_{k-2}&*\\ &&&&1_n&0\\ &&&&& g\end{pmatrix}\mid g\in\GL_n\right\}$$ of $\GL_{kn}$. Denote by $|\det|^{(2-k)n/2}\otimes\psi$ the character of $H$ given by
$$\left(|\det|^{(2-k)n/2}\otimes\psi\right)(x):=|\det(g)|^{(2-k)n/2}\cdot\prod_{i=1}^{k-2}\psi(\Tr(X_i)),$$ where $\det$ is the determinant map. The subgroup$$\GL_n^{\Delta'}:=\left\{\diag(g,\cdots,g,1_n)\mid g\in\GL_n\right\}\subset\GL_{kn},$$ stabilizes the pair $(H,|\det|^{(2-k)n/2}\otimes\psi)$. 
Our first result is the following.

\begin{thm}[Theorem \ref{thm intro1-new}]\label{thm intro1}
For $\pi\in\Irr(\GL_{kn})$ and $\sigma\in\Irr(\GL_n)$, we have
$$\begin{aligned}\Hom_{\GL_n^{\Delta}\cdot N}\left(\Theta(\pi),\sigma\otimes\psi\right) \cong
\Hom_{\GL_n^{\Delta'}\cdot H}\left(\pi^\vee,
\left(|\det|^{(k-2)n/2}\otimes\sigma\right)\otimes\left(|\det|^{(2-k)n/2}\otimes\psi\right)\right).\end{aligned}$$
\end{thm}

\begin{rem}
\begin{enumerate}
\item Compared with the isomorphism \eqref{equ gan} in the case $k=2$, the characters $|\det|^{\pm(k-2)n/2}$ appearing in Theorem \ref{thm intro1} may look unnatural at first sight. They arise in the proof of Theorem \ref{thm intro1} from the explicit action of the Weil representation,  and eventually coincide with half of the modulus character of the subgroup $\GL_n^{\Delta'}\cdot H$.  
\item From the proof of Theorem \ref{thm intro1}, we can also show that
\begin{equation}\label{equ intro}\Hom_{ N}\left(\Theta(\pi),\psi\right)\cong
\Hom_{H}\left(\pi^\vee,
|\det|^{(2-k)n/2}\otimes\psi\right).\end{equation}
\item The subgroups $\GL_n^{\Delta}$ and $\GL_n^{\Delta'}$ act naturally on the left- and right-hand sides of  \eqref{equ intro} respectively. Identifying $\GL_n^{\Delta}$ and $\GL_n^{\Delta'}$ with $\GL_n$, we have
\begin{equation}\label{rem intro2}\Hom_{ N}\left(\Theta(\pi),\psi\right)\cong
\Hom_{H}\left(\pi^\vee,
|\det|^{(2-k)n/2}\otimes\psi\right)\otimes|\det|^{(k-2)n/2} \end{equation}
as representations of $\GL_n$.
\end{enumerate}
\end{rem}

Now assume that $\Theta(\pi)=\theta(\pi)$, or equivalently, $\Theta(\pi)\cong\pi^\vee$. For example, this assumption holds if
$\pi$ is generic. Then $\Theta(\pi^\vee)\cong\pi$, and Theorem \ref{thm intro1} (in the form of \eqref{equ intro}) says that $$\Hom_{N}\left(\pi,\psi\right)\cong
\Hom_{H}\left(\pi,
|\det|^{(2-k)n/2}\otimes\psi\right).$$ If we proceed to study  $$\Hom_{H}\left(\pi,
|\det|^{(2-k)n/2}\otimes\psi\right)=\Hom_{H}\left(\Theta(\pi^\vee),
|\det|^{(2-k)n/2}\otimes\psi\right),$$ we can transfer it to a model attached to a subgroup that is more ``reductive" than $H$. Repeating this process, we eventually arrive at the following subgroup
$$P:=\left\{x=\begin{pmatrix}1_n&0&0&0&0\\ &g_1&*&*&*\\ &&g_2&*&*\\ &&&\ddots&*\\ &&&&g_{k-1} \end{pmatrix}\mid g_i\in\GL_n,\ \forall1\leq i\leq k-1\right\}$$ of $\GL_{kn}$, which can be viewed as the parabolic subgroup of $\GL_{(k-1)n}$ associated to the partition $(n,\cdots,n)$ of $(k-1)n$. Our second result will involve the following characters. Let $\mu$ be the character of $P$ given by
$$\mu(x):=\prod_{i=1}^{k-1}|\det(g_i)|^{\alpha_i},$$ where
\begin{equation}\label{equ char1 intro}\begin{aligned}\left(\alpha_1,\alpha_2,...,\alpha_{k-2},\alpha_{k-1}\right):=\frac{n}{2}\left((k-2)-\sum_{i=1}^{k-2}i,(k-4)-\sum_{i=1}^{k-3}i,\cdots,(4-k)-1,(2-k)\right).\end{aligned}\end{equation}
Let $\chi:=|\det|^\beta$, a character of $\GL_n$, where
\begin{equation}\label{equ char2 intro}\beta:=\frac{1}{2}\sum_{i=1}^{k-1}i(k-1-i)n.\end{equation}
We denote by $\GL_n^{(1)}$ the subgroup $\{\diag(g,1_n,\cdots,1_n)\mid g\in\GL_n\}$ of $\GL_{kn}$.

\begin{thm}\label{thm intro2}
Assume that $\Theta(\pi)=\theta(\pi)$. For $\pi\in\Irr(\GL_{kn})$ and $\sigma\in\Irr(\GL_n)$, we have
$$\Hom_{\GL_n^{\Delta}\cdot N}\left(\pi,\sigma\otimes\psi_{(k-1)}\right)\cong\Hom_{\GL_n^{(1)}\cdot P}\left(\pi,(\chi\otimes\sigma)\otimes\mu\right).$$
\end{thm}

\begin{rem}
\begin{enumerate}
\item The expressions \eqref{equ char1 intro} and \eqref{equ char2 intro}  of the characters $\mu$ and $\chi$ seem to be complicated and unmotivated. In fact, they are combinations of certain modulus characters. See Section \ref{subsec pf of thm2}  for details.
\item As in \eqref{equ intro}, forgetting the $\GL_n$-part, we also have the isomorphism (under the assumption of Theorem \ref{thm intro2})
$$\Hom_{N}\left(\pi,\psi\right)\cong\Hom_{P}\left(\pi,\mu\right).$$
\end{enumerate}
\end{rem}

Our proofs of Theorems \ref{thm intro1} and \ref{thm intro2} rely on the structure of certain coinvariants  of the Weil representation. 
For Theorem  \ref{thm intro1}, we study the twisted Jacquet modules of the Weil representation. The Schrödinger models and the transitions between them play a key role. This approach is standard, and we follow \cite{gz} and \cite{gan19}. For Theorem  \ref{thm intro2}, in addition to  the twisted Jacquet modules, the structure of the coinvariants of the Weil representation attached to certain parabolic subgroups is crucial. In this step, we need the results of \cite{min} on the Jacquet modules of the Weil representation, and the results of \cite{chen} on the description of the big theta lifts.

\subsection{Notation and convention}
Throughout the paper, $F$ denotes a non-archimedean local field of characteristic 0. Let $|\textrm{-}|$  be the normalized absolute value on $F$.
We fix a nontrivial additive character $\psi$ of $F$. 

Let $G$ be a reductive group over $F$. All representations considered are smooth complex representations. Denote by $\Rep(G)$ the category of smooth representations of $G$, by $\Irr(G)$ the set of equivalence classes of irreducible smooth representations, and by 1 or $\BC$ the trivial representation. For $\pi\in\Rep(G)$, let $\pi^\vee$ be its contragredient. 

Let $H$ be a closed subgroup of $G$. Denote by $\delta_H$ the modulus character of $H$.  For $\sigma\in \Rep(H)$, denote by $\uInd_H^G(\sigma)$ the (smooth) induced representation, and by 
$\uind_H^G(\sigma)$ the compactly induced representation. We use $\Ind_H^G$ to denote the normalized induction, and use $\ind_H^G$ to denote the normalized compact induction, i.e.,
$$\Ind_H^G(\sigma)=\uInd_H^G(\delta_H^{1/2}\otimes\sigma),\quad\ind_H^G(\sigma)=\uind_H^G(\delta_H^{1/2}\otimes\sigma).$$ 

Let $\chi$ be a character of $H$. For $\pi\in\Rep(G)$, denote by $\pi_{(H,\chi)}$ the space of coinvariants, i.e., the quotient of $\pi$ modulo the subspace generated by $\{\pi(h)v-\chi(h)v\mid h\in H \}$, where $v$ ranges
over the underlying space of $\pi$. In particular, if $\chi$ is the trivial character, we write $\pi_H$ instead of $\pi_{(H,\chi)}$.
When $G$ is a $p$-adic reductive group and $P=MN$ is a parabolic subgroup of $G$, where $N$ is the unipotent radical and $M$ is a Levi subgroup, we denote by $R_P(\pi)$ the normalized Jacquet module of $\pi$, i.e., $R_P(\pi)=\delta_P^{-1/2}\pi_N$. In this case, the parabolic subgroup opposite to $P$ is denoted by $\overline{P}$.

For a locally profinite space $X$, let $\CS(X)$ be the space of locally constant, compactly supported $\BC$-valued functions on $X$.

\subsection{Acknowledgements}
This work was partially
supported by NSFC Grants 12022106 and 11971223.\section{Preliminaries}

\subsection{Local theta correspondence} We briefly recall the Weil representation and local theta correspondence for type II dual pairs.

\subsubsection{Weil representation} Suppose $W$ and $V$ are vector spaces over $F$, with $\dim W=m$ and $\dim V=n$. Then $$\BW:=\Hom(W,V)\oplus\Hom(V,W)$$ is a symplectic space over $F$,  whose symplectic form is given by
$$\pair{T_1+T_1',T_2+T_2'}=\Tr(T_1T_2')-\Tr(T_1'T_2),\quad T_1,T_2\in\Hom(W,V), T'_1,T'_2\in\Hom(V,W).$$  The pair $(\GL(W),\GL(V))$ is called a reductive dual pair of type II, and admits a Weil representation, denoted $\omega$. The Weil representation can be realized on the space $\CS(\BX)$, where $\BX$ is a Lagrangian of $\BW$. For example, if $\BX=\Hom(W,V)$, the action is given by
$$\left(\omega(h,g) \varphi\right)(X)=|\det g|^{-m/2}\cdot|\det h|^{n/2}\cdot\varphi(g^{-1}Xh),$$ where $(h,g)\in\GL(W)\times\GL(V)$ and $\varphi\in\CS(\Hom(W,V))$. If  $\BX=\Hom(V,W)$, then the action is given by
$$\left(\omega(h,g) \varphi\right)(X)=|\det g|^{m/2}\cdot|\det h|^{-n/2}\cdot\varphi(h^{-1}Xg).$$

\subsubsection{Theta lifts}  For $\pi\in\Irr(\GL(W))$,  its \emph{big theta lift}  $\Theta(\pi)\in\Rep(\GL(V))$ is defined as $$\Theta(\pi):=(\omega\otimes \pi^\vee)_{\GL(W)}.$$  In other words, $\pi\otimes\Theta(\pi)$ is the maximal $\pi$-isotypic quotient of $\omega$. The Howe duality for type II dual pairs, a fundamental result proved by M\'inguez \cite{min}, says that if $\Theta(\pi)$ is nonzero then it has a unique irreducible quotient. We call this quotient the \emph{small theta lift} of $\pi$, and denote it by $\theta(\pi)$. Moreover, M\'inguez \cite[Theorem 1]{min} proved that if $m\leq n$ then $\Theta(\pi)$ is nonzero, and $\theta(\pi)$ can be determined explicitly. In particular, if $m=n$, then $$\theta(\pi)\cong\pi^\vee.$$
 However, the structure of $\Theta(\pi)$ is not completely understood. We refer to \cite{chen}  for the most recent advances on this topic.
If $\pi$ is generic, then $\Theta(\pi)=\theta(\pi)$; this was proved independently by Fang--Sun--Xue \cite{fsx} and Gan \cite{gan19}.

\subsection{Settings} In this subsection, we introduce the notation for the groups, subgroups and their characters considered in this paper. The notation differs from that of the introduction, and looks a little complicated.

\subsubsection{Spaces}
Throughout this paper,  $V$ and $W$ are $F$-vector spaces of dimension $n$. We identify $V$ and $W$ with $F^n$, whose elements are written as column vectors. Identify $\GL(V)$ and $\GL(W)$ with $\GL_n$.
Fix a positive integer $k\geq2$. Denote $$V_{(k)}:=V^{\bigoplus k}.$$ For an integer $1\leq r\leq k$, denote $$V_r:=(0,\cdots,0,\underbrace{V}_{r\textrm{-th position}},0,\cdots,0),$$ which is a subspace of $V_{(k)}$, and  set
$$V^{(r)}:=\bigoplus_{i=1}^r V_i=(\underbrace{V,\cdots,V}_{r\ \textrm{times}},0,\cdots,0),$$
$$V_{(r)}:=\bigoplus_{i=k-r+1}^k V_i=(0,\cdots,0,\underbrace{V,\cdots,V}_{r\ \textrm{times}}).$$
Similarly,  we have the space $W_{(k)}$, and its subspaces $W_r$, $W^{(r)}$ and $W_{(r)}$.

\subsubsection{Lagrangians}
For $r\leq k$, denote $$\BW_r:=\Hom\left(W_{(r)},V_{(r)}\right)\oplus\Hom\left(V_{(r)},W_{(r)}\right).$$ 
Then $$\BX_{r}:=\Hom\left(W_{(r)},V_{(r)}\right)$$ 
is a Lagrangian of $\BW_{r}$. We will also use other Lagrangians. For example,
$$\begin{aligned}\BY_r:&=\Hom\left(V_{k-r+1},W_{(r)}\right)\oplus\Hom\left(W_{(r)},V_{(r-1)}\right)\\
&=\Hom\left(V_{k-r+1},W_{(r)}\right)\oplus\Hom\left(W_{k-r+1},V_{(r-1)}\right)\oplus\BX_{r-1}
\end{aligned}$$ is another Lagrangian of $\BW_{r}$.

\subsubsection{Subgroups and characters} 
Let $1\leq r\leq k-1$. For $A\in\Hom\left(V_{(k-r)},V_r\right)$, denote by $\exp(A)$ the unipotent element of $\GL(V_{(k)})$ such that $$\exp(A)|_{V_i}=\id,\ \forall i\leq r;\quad\exp(A)|_{V_{(k-r)}}=\id+A,$$ where $\id$ stands for the identity map.  Denote by $N_r$ the unipotent subgroup 
$$\left\{\exp(A)\mid A\in\Hom\left(V_{(k-r)},V_r\right)\right\}$$ of $\GL(V_{(k)})$. In terms of matrices, 
\begin{equation}\label{equ Nr}N_r=\left\{x=\begin{pmatrix}1_n&0&\cdots&\cdots&\cdots&0\\ &\ddots&0&\cdots&\cdots&0\\ &&1_n&X_r&*&*\\
&&&1_n&0&0\\ &&&&\ddots&0\\ &&&&&1_n \end{pmatrix}\right\}.\end{equation}
Put $$N_{(r)}:=\prod_{i=1}^rN_i.$$ In particular, $N_{(k-1)}$ is the unipotent radical of the standard parabolic subgroup of $\GL(V_{(k)})$ associated to the partition $(n,n,\cdots,n)$.

For each $r$, we fix an element $B_r\in\Hom\left(V_r,V_{(k-r)}\right)$ given by $$B_r(v):=(v,0,\cdots,0)\in V_{(k-r)}=V_{r+1}\oplus\cdots\oplus V_k.$$ Let $\psi_r$ be the character of $N_r$ defined by $$\psi_r(\exp(A)):=\psi(\Tr_{V_r}(AB_r)),\quad A\in\Hom\left(V_{(k-r)},V_r\right).$$
In terms of the expression \eqref{equ Nr}, we see that
$$\psi_r(x)=\psi(\Tr(X_{r})),\quad x\in N_r.$$ Define
$$\psi_{(r)}:=\otimes_{i=1}^r\psi_i,$$ which is a character of $N_{(r)}$.
If we are working with $\GL(W_{(k)})$, we write $U_r$, $U_{(r)}$ and $D_r$ instead of $N_r$, $N_{(r)}$ and $B_r$, respectively, and still use $\psi_r$ or $\psi_{(r)}$ to denote the relevant characters. 

For $1\leq r\leq k-1$, denote 
$$H_{(r)}:=\GL(W_{(r)})U_{(k-r-1)},$$ which is a subgroup of $\GL(W_{(k)})$. When $r=k-1$, $U_{(0)}$ is interpreted as the trivial group. In particular,
$$H_{(1)}=\GL(W_k)U_{(k-2)},\quad H_{(k-1)}=\GL(W_{(k-1)}).$$
Similar notation $G_{(r)}:=\GL(V_{(r)})N_{(k-r-1)}$ applies to the corresponding subgroup of $\GL(V_{(k)})$.

Let $1\leq r\leq k$. For $g\in\GL_n$, we define the following elements of $\GL(V_{(k)})$:
\begin{itemize}
\item $\Delta^{(r)}(g):=\diag(g,\cdots,g)\in\GL(V^{(r)})$,
\item $\Delta_{(r)}(g):=\diag(g,\cdots,g)\in\GL(V_{(r)})$,
\item $\tau_{(r)}(g):=\diag(g,1,\cdots,1)\in\GL(V_{(r)})$.
\end{itemize}
Similar notation applies to elements of $\GL(W_{(k)})$. Define the following subgroups of $\GL(V_{(k)})\times\GL(W_{(k)})$:
\begin{itemize}
\item $\GL(V)^{\Delta^{(r)}}:=\left\{\Delta^{(r)}(g)\mid g\in\GL_n\right\}\subset\GL(V^{(r)})$,
\item $\GL(V)^{\Delta_{(r)}}:=\left\{\Delta_{(r)}(g)\mid g\in\GL_n\right\}\subset\GL(V_{(r)})$,
\item $G^{\Delta_{(r)}}:=\left\{\left(\Delta_{(r)}(g),\tau_{(r)}(g)\right)\mid g\in\GL_n\right\}\subset\GL(V)^{\Delta_{(r)}}\times\GL(W_{(r)})$,
\item  $G^{\Delta^{(r'),(r)}}:=\left\{\left(\Delta^{(r')}(g),\Delta^{(r)}(g)\right)\mid g\in\GL_n\right\}\subset\GL(V)^{\Delta^{(r')}}\times\GL(W)^{\Delta^{(r)}}$, where $1\leq r'\leq k$.
\end{itemize}
By definition, $G^{\Delta^{(k),(1)}}=G^{\Delta_{(k)}}$.
Note that the stabilizer of $\left(N_{(r-1)},\psi_{(r-1)}\right)$ in $\GL(V^{(r)})$ is the subgroup $\GL(V)^{\Delta^{(r)}}$.

\subsection{Schrödinger models} For $1\leq r,r'\leq k$, denote by $$\omega_{(r),(r')}$$ the Weil representation of the dual pair $(\GL(W_{(r)}),\GL(V_{(r')}))$, which can be realized on the Schrödinger model $\CS\left(\Hom\left(W_{(r)},V_{(r')}\right)\right)$. Denote by $$\omega_{(r), r'}$$ the Weil representation of the dual pair $(\GL(W_{(r)}),\GL(V_{r'}))$, which can be realized on the Schrödinger model $\CS\left(\Hom\left(V_{r'},W_{(r)}\right)\right)$.

Now consider the Weil representation $\omega_{(r),(r)}$. Since
$$\BX_{r}=\Hom\left(W_{(r)},V_{k-r+1}\right)\oplus\Hom\left(W_{(r)},V_{(r-1)}\right)$$ and 
$$\BY_{r}=\Hom\left(V_{k-r+1},W_{(r)}\right)\oplus\Hom\left(W_{(r)},V_{(r-1)}\right),
$$  we can change the model $\CS\left(\BX_r\right)$ of $\omega_{(r),(r)}$ to 
$\CS\left(\BY_r\right)$, using the partial Fourier transform
$$\CF:\CS\left(\BX_r\right)\lra \CS\left(\BY_r\right).$$
The map $\CF$ is explicitly given by
\begin{equation}\label{equ fourier}
\CF(\varphi)(T,S):=\int_{\Hom\left(W_{(r)},V_{k-r+1}\right)}\varphi(T^*,S)\psi\left(\Tr_{V_{k-r+1}}(T^*T)\right)\ \d T^*,\end{equation} where $(T,S)\in\BY_r=\Hom\left(V_{k-r+1},W_{(r)}\right)\oplus\Hom\left(W_{(r)},V_{(r-1)}\right)$.


Let us write down the explicit formulas, which follow readily from \eqref{equ fourier}, for the actions of certain subgroups of $\GL(W_{(r)})$ and $\GL(V_{(r)})$ on $\CS\left(\BY_r\right)$. Let $$\phi\otimes\phi'\in \CS\left(\Hom\left(V_{k-r+1},W_{(r)}\right)\right)\otimes\CS\left(\Hom\left(W_{(r)},V_{(r-1)}\right)\right)=\CS\left(\BY_r\right).$$
\begin{itemize}
\item For $h\in\GL(W_{(r)})$,  we have
\begin{equation}\label{equ action1}\begin{aligned}\left(\omega_{(r),(r)}(h)(\phi\otimes\phi')\right)(T,S)=|\det h|^{-n/2}\cdot\phi(h^{-1}T)\cdot|\det h|^{(r-1)n/2}\cdot\phi'(Sh).
\end{aligned}\end{equation} In other words,
$$\omega_{(r),(r)}|_{\GL(W_{(r)})}
\cong\omega_{(r),k-r+1}|_{\GL(W_{(r)})}\otimes \omega_{(r),(r-1)}|_{\GL(W_{(r)})}.$$
\item For $(g_1,g_2)\in\GL(V_{k-r+1})\times\GL(V_{(r-1)})$, we have 
\begin{equation}\label{equ action2}\left(\omega_{(r),(r)}(g_1,g_2)(\phi\otimes\phi')\right)(T,S)=|\det g_1|^{rn/2}\cdot\phi(Tg_1)\cdot |\det g_2|^{-rn/2}\cdot\phi'(g_2^{-1}S).\end{equation}
In other words, $$\omega_{(r),(r)}|_{\GL(V_{k-r+1})\times\GL(V_{(r-1)})}\cong\omega_{(r), k-r+1}|_{\GL(V_{k-r+1})}\otimes\omega_{(r),(r-1)}|_{\GL(V_{(r-1)})}.$$
\item For $A\in\Hom\left(V_{(r-1)},V_{k-r+1}\right)$, we have  \begin{equation}\label{equ action3}\left(\omega_{(r),(r)}(\exp(A))(\phi\otimes \phi')\right)(T,S)=\psi\left(\Tr_{V_{k-r+1}}(AST)\right)\cdot\phi(T)\phi'(S).\end{equation}
\end{itemize}

\subsection{Some isomorphisms}

We will frequently use the following  lemmas.

\begin{lem}\label{lem coinvariant}\cite{mvw}
Let $H$ be a locally profinite group, and $\chi$ a character of $H$. Suppose that $H=H_1H_2$, where $H_1$ and $H_2$ are subgroups of $H$. Assume that $H_1$ normalizes $H_2$. Then, for $\pi\in\Rep(H)$,
$$\pi_{(H,\chi)}\cong \left(\pi_{(H_2,\chi)}\right)_{(H_1,\chi)}.$$
\end{lem}

\begin{lem}\label{lem ind}\cite[Lemme 1.3]{min}
Let $G$ be a locally profinite group, and $\CA$ a $G$-space. Suppose $\CB$ is a closed subspace of $\CA$ such that $G\cdot \CB=\CA$. Let $H=\Stab_G(\CB)$. Denote by $\rho$ the natural representation of $G$ on $\CS(\CA)$ and $\rho_H$ the natural representation of $H$ on $\CS(\CB)$. Then
$$\rho\cong\mathrm{uind}_H^G(\rho_H).$$
\end{lem}

\begin{lem}[Frobenius reciprocity]\label{lem frob}
Let $H$ be a closed subgroup of a locally profinite group $G$. For $\pi\in\Rep(G)$ and $\sigma\in\Rep(H)$, we have
$$\Hom_G\left(\pi,\Ind_H^G(\sigma)\right)\cong\Hom_H\left(\pi,\delta_H^{1/2}\otimes\sigma\right),$$
and
$$\Hom_G\left(\ind_H^G(\sigma),\pi^\vee\right)\cong\Hom_H\left(\pi,\delta_H^{1/2}\otimes\sigma^\vee\right).$$
\end{lem}

\begin{lem}[Bernstein's Frobenius reciprocity]\label{lem bernstein}
Let $P=MN$ be a parabolic subgroup of a $p$-adic reductive group $G$. For $\pi\in\Rep(G)$ and $\sigma\in\Rep(M)$, we have
$$\Hom_G\left(\Ind_P^G(\sigma),\pi\right)\cong\Hom_M\left(\sigma,R_{\overline{P}}(\pi)\right).$$
\end{lem}


\section{Proof of Theorem \ref{thm intro1}}

In terms of the notation introduced in Section 2, Theorem \ref{thm intro1} can be rephrased as: 

\begin{thm}\label{thm intro1-new}
For $\pi\in\Irr(\GL(W_{(k)}))$ and $\sigma\in\Irr(\GL_n)$, we have
$$\begin{aligned}&\Hom_{\GL(V)^{\Delta^{(k)}}\cdot N_{(k-1)}}\left(\Theta(\pi),\sigma\otimes\psi_{(k-1)}\right)\\\cong&\Hom_{\GL(W)^{\Delta^{(k-1)}}\cdot H_{(1)}}\left(\pi^\vee,
\left(|\det\nolimits_W|^{(k-2)n/2}\otimes\sigma\right)\otimes\left(|\det\nolimits_{W_k}|^{(2-k)n/2}\otimes\psi_{(k-2)}\right)\right).\end{aligned}$$
\end{thm}

We will prove Theorem \ref{thm intro1-new} at the end of this section. Note that, for $g\in\GL(W)$, $h\in\GL(W_k)=\GL(W_{(1)})$, and $u\in U_{(k-2)}$, we have
$$\delta_{\GL(W)^{\Delta^{(k-1)}}\cdot H_{(1)}}\left(\Delta^{(k-1)}(g)hu\right)=|\det\nolimits_W(g)|^{(k-2)n}\cdot
|\det\nolimits_{W_k}(h)|^{(2-k)n}.$$

\subsection{Twisted Jacquet modules of the Weil representation} Our proof of Theorem \ref{thm intro1-new} relies on Proposition \ref{key prop1} stated below, which describes the structure of the twisted Jacquet module  $\left(\omega_{(k),(k)}\right)_{\left(N_{(r)},\psi_{(r)}\right)}$ as an induced representation. 

Consider the Weil representation 
$\omega_{(k-r),(k-r-1)}$ of $\GL(V_{(k-r-1)})\times\GL(W_{(k-r)})$. Via  the isomorphism
\begin{equation}\label{equ gp-isom1}\begin{tikzcd}[row sep=0.1cm]
	{G^{\Delta^{(k),(r)}}} & {\GL(V)^{\Delta_{(k-r-1)}}} \\
	{\left(\Delta^{(k)}(g),\Delta^{(r)}(g)\right)} & {\Delta_{(k-r-1)}(g)},
	\arrow["\sim", from=1-1, to=1-2]
	\arrow[maps to, from=2-1, to=2-2]
\end{tikzcd}\end{equation}
and the restriction from $\GL(V_{(k-r-1)})$ to $\GL(V)^{\Delta_{(k-r-1)}}$,  we can view $\omega_{(k-r),(k-r-1)}$ as a representation of $G^{\Delta^{(k),(r)}}\times\GL(W_{(k-r)})$.  Combining this with the character $\psi_{(r-1)}$ of $U_{(r-1)}$, we obtain a representation $\omega_{(k-r),(k-r-1)}\otimes\psi_{(r-1)}^{-1}$ of $G^{\Delta^{(k),(r)}}\cdot H_{(k-r)}$.
When $r=k-1$, we interpret  $\omega_{(1),(0)}$ as the trivial representation.

\begin{prop}\label{key prop1}
For $1\leq r\leq k-1$, we have
$$\left(\omega_{(k),(k)}\right)_{\left(N_{(r)},\psi_{(r)}\right)}\cong\ind_{G^{\Delta^{(k),(r)}}\cdot H_{(k-r)}}^{\GL(V)^{\Delta^{(k)}}\times\GL(W_{(k)})}\left(\omega_{(k-r),(k-r-1)}\otimes\psi_{(r-1)}^{-1}\right).$$
 In particular, when $r=k-1$, we have
$$\left(\omega_{(k),(k)}\right)_{\left(N_{(k-1)},\psi_{(k-1)}\right)}\cong\ind_{G^{\Delta^{(k),(k-1)}}\cdot H_{(1)}}^{\GL(V)^{\Delta^{(k)}}\times\GL(W_{(k)})}\left(\BC\otimes\psi_{(k-2)}^{-1}\right).$$ 
\end{prop}

We will prove Proposition \ref{key prop1} in the next subsection. In the course of the proof we also need another description of 
$\left(\omega_{(k),(k)}\right)_{\left(N_{(r)},\psi_{(r)}\right)}$, which is Proposition \ref{key prop1'} stated below. Let us introduce more notation. For each $1\leq r\leq k-1$, denote
$$H'_{(r)}:=\GL(W_{(r)})U_{(k-r)}=H_{(r)}U_{k-r}.$$  
We define a linear map
\begin{equation}\label{equ S_r}\begin{tikzcd}[row sep=0.1cm]
	S_{r}:W_{r} & V_{(k-r)}\\
	w & (w,0,\cdots,0).
	\arrow[ from=1-1, to=1-2]
	\arrow[maps to, from=2-1, to=2-2]
\end{tikzcd}\end{equation}
Suppose that we realize the Weil representation $\omega_{(k-r),(k-r)}$ on the model $$\CS\left(\BX_{k-r}\right)=\CS\left(\Hom\left(W_{(k-r)},V_{(k-r)}\right)\right).$$ There is an auxiliary action of  $U_r$ on  $\CS\left(\BX_{k-r}\right)$ given by
\begin{equation}\label{equ U-action}(\exp(C)\varphi)(X):=\varphi(X+S_rC),\end{equation} 
where  $C\in\Hom\left(W_{(k-r)},W_r\right)$.
By abuse of the notation, we denote this action of $U_r$ on $\CS(\BX_{k-r})$ by $\omega_{(k-r),(k-r)}$. 
Via  the isomorphism
\begin{equation}\label{equ gp-isom2}\begin{tikzcd}[row sep=0.1cm]
	{G^{\Delta^{(k),(r)}}} & {\GL(V)^{\Delta_{(k-r)}}} \\
	{\left(\Delta^{(k)}(g),\Delta^{(r)}(g)\right)} & {\Delta_{(k-r)}(g)},
	\arrow["\sim", from=1-1, to=1-2]
	\arrow[maps to, from=2-1, to=2-2]
\end{tikzcd}\end{equation}
and the restriction from $\GL(V_{(k-r)})$ to $\GL(V)^{\Delta_{(k-r)}}$,  we view $\omega_{(k-r),(k-r)}$ as a representation of $G^{\Delta^{(k),(r)}}\times\GL(W_{(k-r)})$. 
Combining this with the character $\psi_{(r-1)}$ of $U_{(r-1)}$, we obtain a representation $\omega_{(k-r),(k-r)}\otimes\psi_{(r-1)}^{-1}$ of $G^{\Delta^{(k),(r)}}\cdot H'_{(k-r)}$.

\begin{prop}\label{key prop1'}
For $1\leq r\leq k-1$, we have
$$\left(\omega_{(k),(k)}\right)_{\left(N_{(r)},\psi_{(r)}\right)}\cong\ind_{G^{\Delta^{(k),(r)}}\cdot H'_{(k-r)}}^{\GL(V)^{\Delta^{(k)}}\times\GL(W_{(k)})}\left(\omega_{(k-r),(k-r)}\otimes\psi_{(r-1)}^{-1}\right).$$
\end{prop}

\begin{rem} We record the formulas for the modulus characters  involved in the normalized inductions in Proposition \ref{key prop1} and Proposition \ref{key prop1'}.
It is easy to see that
\begin{equation}\label{equ modulus H}
\delta_{G^{\Delta^{(k),(r)}}\cdot H_{(k-r)}}=|\det\nolimits_{V_{(k)}}|^{(r-1)n}\otimes|\det\nolimits_{W_{(k)}}|^{(1-r)n}.
\end{equation}
Explicitly, for $\left(\Delta^{(k)}(g),\Delta^{(r)}(g)\right)\in G^{\Delta^{(k),(r)}}$ and $h\in\GL(W_{(k-r)})$,
$$\delta_{G^{\Delta^{(k),(r)}}\cdot H_{(k-r)}}\left(\left(\Delta^{(k)}(g),\Delta^{(r)}(g)\right)\cdot h \right)=|\det\nolimits_V(g)|^{(k-r)(r-1)n}\cdot |\det\nolimits_{W_{(k-r)}}(h)|^{(1-r)n}.$$
It is also easy to check that
\begin{equation}\label{equ modulus H'}
\delta_{G^{\Delta^{(k),(r)}}\cdot H'_{(k-r)}}=|\det\nolimits_{V_{(k)}}|^{rn}\otimes |\det\nolimits_{W_{(k)}}|^{-rn}.
\end{equation}
\end{rem}

\subsection{Proof of Proposition \ref{key prop1} and Proposition \ref{key prop1'} } 
It follows from Lemma \ref{lem coinvariant} that
$$\left(\omega_{(k),(k)}\right)_{\left(N_{(r)},\psi_{(r)}\right)}\cong\left(\left(\omega_{(k),(k)}\right)_{\left(N_1,\psi_1\right)}\right)_{\left(N_2\cdots N_r,\psi_2\cdots\psi_r\right)}.$$ We will induct on $r$ to prove Proposition \ref{key prop1} and Proposition \ref{key prop1'} simultaneously.  Let us first consider the case when $r=1$. Recall that $G^{\Delta^{(k),(1)}}=G^{\Delta_{(k)}}$.

\begin{lem}\label{lem 1st isom} We have
$$\left(\omega_{(k),(k)}\right)_{\left(N_{1},\psi_1\right)}\cong \ind_{\left(G^{\Delta_{(k)}}\times\GL(W_{(k-1)})\right)U_1}^{\GL(V)^{\Delta^{(k)}}\times\GL(W_{(k)})}\left(\omega_{(k-1),(k-1)}\right),$$ 
and $$\left(\omega_{(k),(k)}\right)_{\left(N_{1},\psi_1\right)}\cong \ind_{G^{\Delta_{(k)}}\times\GL(W_{(k-1)})}^{\GL(V)^{\Delta^{(k)}}\times\GL(W_{(k)})}\left(\omega_{(k-1),(k-2)}\right).$$
\end{lem}

\begin{proof} We use the model $\CS(\BY_k)$ of the Weil representation $\omega_{(k),(k)}$. Since the twisted Jacquet functor is exact, by \eqref{equ action3}, we have
$$\left(\omega_{(k),(k)}\right)_{\left(N_{1},\psi_1\right)}=\CS(\BY_k)_{\left(N_{1},\psi_1\right)}\cong\CS(\CA_1),$$ where 
\begin{equation}\label{equ defn-ca}\CA_1:=\left\{(T,S)\mid \Tr_{V_1}(AST)=\Tr_{V_1}(AB_1), \forall A\in\Hom\left(V_{(k-1)},V_1\right) \right\}\end{equation} is a closed subset of $\BY_k$. It is easy to see that the condition in \eqref{equ defn-ca} is equivalent to 
\begin{equation}\label{equ st}ST=B_1.\end{equation} 
The group $\GL(V)^{\Delta^{(k)}}\times\GL(W_{(k)})$ acts on $\CA_1$ via the natural geometric action
$$(\Delta^{(k)}(g),h)\cdot(T,S)=(hT\tau_{(k)}(g)^{-1},\Delta_{(k-1)}(g)Sh^{-1}),\quad g\in\GL(V),\ h\in\GL(W_{(k)}).$$  Denote by $\rho$ the natural representation of $\GL(V)^{\Delta^{(k)}}\times\GL(W_{(k)})$ on $\CS(\CA_1)$ induced by this geometric action. By \eqref{equ action1}, we see that \begin{equation}\label{equ isom-pf1}\left(\omega_{(k),(k)}\right)_{\left(N_{1},\psi_1\right)}\cong |\det\nolimits_{V_{(k)}}|^{(2-k)n/2}\otimes|\det\nolimits_{W_{(k)}}|^{(k-2)n/2}\otimes\rho.\end{equation}

The condition \eqref{equ st} forces $\rank(T)=n$. Consequently, if  $(T,S)$ and $(T',S')$ satisfy the condition $ST=S'T'=B_1$,  there exists $h\in\GL(W_{(k)})$ such that 
\begin{equation}\label{equ trans-T} T'=hT. \end{equation}
Let us choose a specific element of $\CA_1$.  Set
$$\begin{tikzcd}[row sep=0.1cm]
	T_{1}: V_{1} & W_{(k)}\\
	v & (v,0,\cdots,0).
	\arrow[ from=1-1, to=1-2]
	\arrow[maps to, from=2-1, to=2-2]
\end{tikzcd}$$
Recall the map 
 $S_1:W_{1}\ra V_{(k-1)}$ defined in  \eqref{equ S_r}. View $S_1$ as an element of $\Hom\left(W_{(k)},V_{(k-1)}\right)$ via the natural projection $W_{(k)}\twoheadrightarrow W_{1}$. Then $S_1T_1=B_1$. Note that the image of $S_1$ is $V_{2}$. 
Put $$\CB_1:=\left\{(T,S)\in\CA_1\mid T=T_1\right\}.$$ 
Then $(T_1,S_1)\in\CB_1$. By \eqref{equ trans-T}, we see that
$$\left(\GL(V)^{\Delta^{(k)}}\times\GL(W_{(k)})\right)\cdot\CB_1=\CA_1.$$ It is easy to check that
\begin{equation}\label{equ def-cb}\CB_1=\left\{(T_1,S_1+X)\mid X\in\BX_{k-1} \right\},\end{equation} 
and
$$\Stab_{\left(\GL(V)^{\Delta^{(k)}}\times\GL(W_{(k)})\right)}(\CB_1)=\left(G^{\Delta_{(k)}}\times\GL(W_{(k-1)})\right)U_1.$$ Denote by $\rho'$ the natural action of $\left(G^{\Delta_{(k)}}\times\GL(W_{(k-1)})\right)U_1$ on  $\CS(\CB_1)$. By Lemma \ref{lem ind}, we have  
\begin{equation}\label{equ ind-isom}\rho\cong\mathrm{uind}_{\left(G^{\Delta_{(k)}}\times\GL(W_{(k-1)})\right)U_1}^{\GL(V)^{\Delta^{(k)}}\times\GL(W_{(k)})}(\rho').\end{equation} 

We now analyze the representation $\rho'$.
Consider the map
 \begin{equation}\label{equ isom-map}\begin{tikzcd}[row sep=0.1cm]
	\CS(\CB_1)& \CS\left(\BX_{k-1}\right)\\
	\varphi & \wt{\varphi},
	\arrow[ from=1-1, to=1-2]
	\arrow[maps to, from=2-1, to=2-2]
\end{tikzcd}\end{equation} which is defined by
$$\wt{\varphi}(X):=\varphi(T_1,S_1+X),\quad X\in\BX_{k-1}.$$
Recall the isomorphism \eqref{equ gp-isom2} $$G^{\Delta_{(k)}}\cong\GL(V)^{\Delta_{(k-1)}}.$$
By \eqref{equ action1}, the map \eqref{equ isom-map} establishes an isomorphism  
\begin{equation}\label{equ isom-pf2}\rho'\cong|\det\nolimits_{V_{(k-1)}}|^{(k-1)n/2}\otimes|\det\nolimits_{W_{(k-1)}}|^{(1-k)n/2}\otimes\omega_{(k-1),(k-1)}\end{equation} between the representations of $G^{\Delta_{(k)}}\times\GL(W_{(k-1)})$ and  $\GL(V)^{\Delta_{(k-1)}}\times\GL(W_{(k-1)})$.
Under the map \eqref{equ isom-map},
the action of $U_1$ on $\CS(\CB_1)$ is transferred to $\CS\left(\BX_{k-1}\right)$, and is given by
\begin{equation}\label{equ U-action1}(\exp(C)\wt{\varphi})(X)=\wt{\varphi}(X+S_1C),\quad C\in\Hom\left(W_{(k-1)},W_1\right).\end{equation}
Note that the action \eqref{equ U-action1} is precisely the one defined by \eqref{equ U-action}.

Combining \eqref{equ isom-pf1}, \eqref{equ ind-isom} and \eqref{equ isom-pf2}, we obtain 
$$\begin{aligned}\left(\omega_{(k),(k)}\right)_{\left(N_{1},\chi_1\right)}&\cong |\det\nolimits_{V_{(k)}}|^{(2-k)n/2}\otimes|\det\nolimits_{W_{(k)}}|^{(k-2)n/2}\otimes\rho\\
&\cong\mathrm{uind}_{\left(G^{\Delta_{(k)}}\times\GL(W_{(k-1)})\right)U_1}^{\GL(V)^{\Delta^{(k)}}\times\GL(W_{(k)})}\left(|\det\nolimits_{V_{(k-1)}}|^{(2-k)n/2}\otimes|\det\nolimits_{W_{(k-1)}}|^{(k-2)n/2}\otimes\rho'\right)\\
&\cong\mathrm{uind}_{\left(G^{\Delta_{(k)}}\times\GL(W_{(k-1)})\right)U_1}^{\GL(V)^{\Delta^{(k)}}\times\GL(W_{(k)})}\left(|\det\nolimits_{V_{(k-1)}}|^{n/2}\otimes|\det\nolimits_{W_{(k-1)}}|^{-n/2}\otimes\omega_{(k-1),(k-1)}\right)\\
&=\ind_{\left(G^{\Delta_{(k)}}\times\GL(W_{(k-1)})\right)U_1}^{\GL(V)^{\Delta^{(k)}}\times\GL(W_{(k)})}\left(\omega_{(k-1),(k-1)}\right),
\end{aligned}$$
which finishes the proof of the first statement of Lemma \ref{lem 1st isom}.

To prove the second isomorphism in Lemma \ref{lem 1st isom}, we shrink $\CB_1$. Let us consider
$$\CB_1':=\left\{(T_1,S_1+Y)\mid Y\in\Hom\left(W_{(k-1)},V_{(k-2)}\right)\right\}\subset\CB_1,$$
where we view $\Hom\left(W_{(k-1)},V_{(k-2)}\right)$ as a subspace of 
$\BX_{k-1}$ via the inclusion $V_{(k-2)}\incl V_{(k-1)}$. A direct computation shows that 
$$\left(\GL(V)^{\Delta^{(k)}}\times\GL(W_{(k)})\right)\cdot\CB'_1=\CA_1,$$ and
$$\Stab_{ \left(\GL(V)^{\Delta^{(k)}}\times\GL(W_{(k)})\right)}(\CB'_1)=G^{\Delta_{(k)}}\times\GL(W_{(k-1)}).$$
We still denote by $\rho'$ the natural action of $G^{\Delta_{(k)}}\times\GL(W_{(k-1)})$ on  
$\CS(\CB'_1)$.  
Under the  isomorphism \eqref{equ gp-isom1}
$$G^{\Delta_{(k)}}\cong\GL(V)^{\Delta_{(k-2)}},$$ as in \eqref{equ isom-pf2},  $$\rho'\cong|\det\nolimits_{V_{(k-2)}}|^{(k-1)n/2}\otimes|\det\nolimits_{W_{(k-1)}}|^{(2-k)n/2}\otimes\omega_{(k-1),(k-2)}$$
as representations of $G^{\Delta_{(k)}}\times\GL(W_{(k-1)})$ and  $\GL(V)^{\Delta_{(k-2)}}\times\GL(W_{(k-1)})$.
Then we obtain
$$\left(\omega_{(k),(k)}\right)_{\left(N_{1},\chi_1\right)}\cong\ind_{G^{\Delta_{(k)}}\times\GL(W_{(k-1)})}^{\GL(V)^{\Delta^{(k)}}\times\GL(W_{(k)})}\left(\omega_{(k-1),(k-2)}\right).$$
\end{proof}

We now assume that the statements of Proposition \ref{key prop1} and Proposition \ref{key prop1'}  hold for $r$, and show that they hold for $r+1$.
It follows from Lemma \ref{lem coinvariant}  that
$$\left(\omega_{(k),(k)}\right)_{\left(N_{(r+1)},\psi_{(r+1)}\right)}\cong\ind_{G^{\Delta^{(k),(r)}}\cdot H'_{(k-r)}}^{\GL(V)^{\Delta^{(k)}}\times\GL(W_{(k)})}\left(\left(\omega_{(k-r),(k-r)}\right)_{(N_{r+1},\psi_{r+1})}\otimes\psi_{(r-1)}^{-1}\right).$$ Recall that $$H'_{(k-r)}=\GL(W_{(k-r)})U_{(r)}=\GL(W_{(k-r)})U_rU_{(r-1)}.$$ Also recall that the action of
$G^{\Delta^{(k),(r)}}$ on $\omega_{(k-r),(k-r)}$ is via that of $\GL(V)^{\Delta_{(k-r)}}$, and the action of $U_r$ on $\omega_{(k-r),(k-r)}$ is given by \eqref{equ U-action}.

We analyze the twisted Jacquet module $\left(\omega_{(k-r),(k-r)}\right)_{(N_{r+1},\psi_{r+1})}$, as a representation of $$\left(\GL(V)^{\Delta_{(k-r)}}\times\GL(W_{(k-r)})\right)U_r$$ in two steps. 
 The first step is to analyze the action of $\GL(V)^{\Delta_{(k-r)}}\times\GL(W_{(k-r)})$. It follows the same line as the proof of Lemma \ref{lem 1st isom}, and will be discussed briefly.  
Recall that
$$\BX_{k-r}=\Hom\left(W_{(k-r)},V_{(k-r)}\right),\quad \BY_{k-r}=\Hom\left(V_{r+1},W_{(k-r)}\right)\oplus\Hom\left(W_{(k-r)},V_{(k-r-1)}\right),$$ and the intertwining operator $\CF:\CS\left(\BX_{k-r}\right)\ra \CS\left(\BY_{k-r}\right)$ between the models of $\omega_{(k-r),(k-r)}$ is given by the partial Fourier transform, i.e.,
$$\CF(\varphi)(T,S)=\int_{\Hom\left(W_{(k-r)},V_{r+1}\right)}\varphi(T^*,S)\psi(\Tr_{V_{r+1}}(T^*T))\ \d T^*.$$ 
By \eqref{equ action3}, we have
$$\left(\omega_{(k-r),(k-r)}\right)_{\left(N_{r+1},\chi_{r+1}\right)}=\CS(\BY_{k-r})_{(N_{r+1},\chi_{r+1})}\cong\CS(\CA_{r+1}),$$  where $$\CA_{r+1}:=\left\{(T,S)\in\BY_{k-r}\mid ST=B_{r+1}\right\}.$$
Denote by $\rho$ the natural representation of $\GL(V)^{\Delta_{(k-r)}}\times\GL(W_{(k-r)})$ on $\CS(\CA_{r+1})$ induced by the geometric action of 
$\GL(V)^{\Delta_{(k-r)}}\times\GL(W_{(k-r)})$ on $\CA_{r+1}$. Then, by \eqref{equ action1}, we have
\begin{equation}\label{equ isom-pf1'}\left(\omega_{(k-r),(k-r)}\right)_{\left(N_{r+1},\psi_{r+1}\right)}\cong|\det\nolimits_{V_{(k-r)}}|^{(2+r-k)n/2}\otimes|\det\nolimits_{W_{(k-r)}}|^{(k-r-2)n/2}\otimes\rho.\end{equation}
Let
$$\begin{tikzcd}[row sep=0.1cm]
	T_{r+1}: V_{1} & W_{(k-r)}\\
	v & (v,0,\cdots,0),
	\arrow[ from=1-1, to=1-2]
	\arrow[maps to, from=2-1, to=2-2]
\end{tikzcd}$$
and $$\begin{aligned}\CB_{r+1}:=\left\{(T,S)\in\CA_{r+1}\mid T=T_{r+1}\right\}
=\left\{(T_{r+1},S_{r+1}+X)\mid X\in \BX_{k-r-1}\right\},\end{aligned}$$ where $S_{r+1}$ is defined by \eqref{equ S_r}. It is easy to check that
$$\left(\GL(V)^{\Delta_{(k-r)}}\times\GL(W_{(k-r)})\right)\cdot\CB_{r+1}=\CA_{r+1},$$ and
$$\Stab_{\left(\GL(V)^{\Delta_{(k-r)}}\times\GL(W_{(k-r)})\right)}(\CB_{r+1})=\left(G^{\Delta_{(k-r)}}\times\GL(W_{(k-r-1)})\right)U_{r+1}.$$
Let $\rho'$ be the representation of $\left(G^{\Delta_{(k-r)}}\times\GL(W_{(k-r-1)})\right)U_{r+1}$ on  
$\CS(\CB_{r+1})$ induced by the natural geometric action of
$\left(G^{\Delta_{(k-r)}}\times\GL(W_{(k-r-1)})\right)U_{r+1}$ on  
$\CB_{r+1}$. Then we have
\begin{equation}\label{equ ind-isom'}\rho\cong\mathrm{uind}_{\left(G^{\Delta_{(k-r)}}\times\GL(W_{(k-r-1)})\right)U_{r+1}}^{\GL(V)^{\Delta_{(k-r)}}\times\GL(W_{(k-r)})}(\rho').\end{equation} As in \eqref{equ isom-pf2},
\begin{equation}\label{equ isom-pf2'}\rho'\cong|\det\nolimits_{V_{(k-r-1)}}|^{(k-r-1)n/2}\otimes|\det\nolimits_{W_{(k-r-1)}}|^{(1+r-k)n/2}\otimes\omega_{(k-1),(k-1)},\end{equation} where $G^{\Delta_{(k-r)}}$ and $U_{r+1}$ act on the right hand side via the isomorphism
$$\begin{tikzcd}[row sep=0.1cm]
	G^{\Delta_{(k-r)}} & {\GL(V)^{\Delta_{(k-r-1)}}} \\
	{\left(\Delta_{(k-r)}(g),\tau_{(k-r)}(g)\right)} & {\Delta_{(k-r-1)}(g)}
	\arrow["\sim", from=1-1, to=1-2]
	\arrow[maps to, from=2-1, to=2-2]
\end{tikzcd}$$
and \eqref{equ U-action} respectively.

The second step is to analyze the action of $U_r$,  a subgroup of $H'_{(k-r)}$, on  $\left(\omega_{(k-r),(k-r)}\right)_{(N_{r+1},\psi_{r+1})}$.
The action \eqref{equ U-action} of $U_r$ on $\CS\left(\BX_{k-r}\right)$ is transferred to $\CS\left(\BY_{k-r}\right)$ by the rule
$$\exp(C)\CF(\varphi):=\CF(\exp(C)\varphi),\quad C\in\Hom(W_{(k-r)},W_r).$$ If $\phi=\CF(\varphi)$, we obtain
\begin{equation}\label{equ U-action 2}
\begin{aligned}(\exp(C)\phi)(T,S)&=\int_{\Hom(W_{(k-r)},V_{r+1})}\varphi(T^*+S_rC,S)\psi(\Tr_{V_{r+1}}(T^*T))\ \d T^*\\
&=\int_{\Hom(W_{(k-r)},V_{r+1})}\varphi(T^*,S)\psi(\Tr_{V_{r+1}}(T^*T-S_rCT))\ \d T^*\\
&=\psi^{-1}(\Tr_{V_{r+1}}(S_rCT))\int_{\Hom(W_{(k-r)},V_{r+1})}\varphi(T^*,S)\psi(\Tr_{V_{r+1}}(T^*T))\ \d T^*\\
&=\psi^{-1}(\Tr_{W_r}(CTS_r))\cdot\phi(T,S).
\end{aligned}
\end{equation}
In particular, $U_r$ acts on $\CS(\CA_{r+1})$ by \eqref{equ U-action 2}.
Note that the restriction map
 $\CS(\CA_{r+1})\ra\CS(\CB_{r+1})$ is $U_r$-equivariant. Hence $U_r$ acts on $\CS(\CB_{r+1})$ via the character
$$\psi_r^{-1}(\exp(C)):=\psi^{-1}(\Tr_{W_r}(CT_{r+1}S_r)).$$ Then $\CS(\CB_{r+1})$ is a representation of $\left(G^{\Delta_{(k-r)}}\times\GL(W_{(k-r-1)})\right)U_{r+1}U_r$, isomorphic to $\rho'\otimes\psi_r^{-1}$.
We still denote by $\rho$ the resulting representation of $\left(\GL(V)^{\Delta_{(k-r)}}\times\GL(W_{(k-r)})\right)U_r$ on $\CS(\CA_{r+1})$. 
 Following the proof of \cite[Lemme 1.3]{min}, it is routine to check that
\begin{equation}\label{equ isom-pf3}\begin{tikzcd}[row sep=0.1cm]
	\rho & \mathrm{uind}_{\left(G^{\Delta_{(k-r)}}\times\GL(W_{(k-r-1)})\right)U_{r+1}U_r}^{\left(\GL(V)^{\Delta_{(k-r)}}\times\GL(W_{(k-r)})\right)U_r}\left(\rho'\otimes\psi_r^{-1}\right) \\
	\phi &  \left(gu\mapsto\rho(gu))\phi|_{\CB_{r+1}}\right)	\arrow[from=1-1, to=1-2]
	\arrow[maps to, from=2-1, to=2-2]
\end{tikzcd}\end{equation}
 is an isomorphism of representations, where $g\in \GL(V)^{\Delta_{(k-r)}}\times\GL(W_{(k-r)})$ and  $u\in U_r$. The inverse of the above map is given as follows. Let $$f\in \mathrm{uind}_{\left(G^{\Delta_{(k-r)}}\times\GL(W_{(k-r-1)})\right)U_{r+1}U_r}^{\left(\GL(V)^{\Delta_{(k-r)}}\times\GL(W_{(k-r)})\right)U_r}\left(\rho'\otimes\psi_r^{-1}\right).$$ For $(T,S)\in\CA_{r+1}$, choose $g\in\GL(V)^{\Delta_{(k-r)}}\times\GL(W_{(k-r)})$ and $(T_{r+1},S')\in\CB_{r+1}$ such that $(T,S)=g\cdot(T_{r+1},S')$. Then define the image $\phi_f\in\CS(\CA_{r+1})$ of $f$ by $$\phi_f(T,S)=f(g)(T_{r+1},S').$$

Combining the isomorphisms \eqref{equ isom-pf1'}, \eqref{equ ind-isom'}, \eqref{equ isom-pf2'} and \eqref{equ isom-pf3}, we see that $\left(\omega_{(k-r),(k-r)}\right)_{\left(N_{r+1},\psi_{r+1}\right)}$ is isomorphic to
$$\mathrm{uind}_{\left(G^{\Delta_{(k-r)}}\times\GL(W_{(k-r-1)})\right)U_{r+1}U_r}^{\left(\GL(V)^{\Delta_{(k-r)}}\times\GL(W_{(k-r)})\right)U_r}\left(|\det\nolimits_{V_{(k-r-1)}}|^{n/2}\otimes|\det\nolimits_{W_{(k-r-1)}}|^{-n/2}\otimes\omega_{(k-r-1),(k-r-1)}\otimes\psi_r^{-1}\right).$$
Note that, under the identification \eqref{equ gp-isom2} of  
$\GL(V)^{\Delta_{(k-r)}}$ with $G^{\Delta^{(k),(r)}}$, the subgroup $G^{\Delta_{(k-r)}}$ of $\GL(V)^{\Delta_{(k-r)}}\times\GL(W_{(k-r)})$ is identified with the subgroup $G^{\Delta^{(k),(r+1)}}$ of $\GL(V)^{\Delta^{(k)}}\times\GL(W_{(k)})$.
Hence $\left(\omega_{(k),(k)}\right)_{\left(N_{(r+1)},\psi_{(r+1)}\right)}$ is isomorphic to
$$\begin{aligned}
&\uind_{\left(G^{\Delta^{(k),(r+1)}}\times\GL(W_{(k-r-1)})\right)U_{r+1}U_rU_{(r-1)}}^{
\GL(V)^{\Delta^{(k)}}\times\GL(W_{(k)})}\left(\begin{array}{c}\delta_{G^{\Delta^{(k),(r)}}\cdot H'_{(k-r)}}^{1/2}\otimes|\det\nolimits_{V_{(k-r-1)}}|^{n/2}\otimes|\det\nolimits_{W_{(k-r-1)}}|^{-n/2}\\ \otimes\omega_{(k-r-1),(k-r-1)}\otimes\psi_{r}^{-1}\otimes\psi_{(r-1)}^{-1}\end{array}\right)\\
=&\ind_{G^{\Delta^{(k),(r+1)}}\cdot H'_{(k-r-1)}}^{
\GL(V)^{\Delta^{(k)}}\times\GL(W_{(k)})}\left(\omega_{(k-r-1),(k-r-1)}\otimes\psi_{(r)}^{-1}\right).
\end{aligned}$$
Therefore 
the statement of Proposition \ref{key prop1'}  holds for $r+1$.

For Proposition \ref{key prop1}, as in the proof of Lemma \ref{lem 1st isom}, replacing $\CB_{r+1}$ by $$\CB'_{r+1}:=\left\{(T_{r+1},S_{r+1}+X)\mid X\in \Hom\left(W_{(k-r-1)},V_{(k-r-2)}\right)\right\},$$ the previous argument yields the isomorphism
$$\left(\omega_{(k),(k)}\right)_{\left(N_{(r+1)},\psi_{(r+1)}\right)}\cong\ind_{G^{\Delta^{(k),(r+1)}}\cdot H_{(k-r-1)}}^{
\GL(V)^{\Delta^{(k)}}\times\GL(W_{(k)})}\left(\omega_{(k-r-1),(k-r-2)}\otimes\psi_{(r)}^{-1}\right).$$
Thus 
the statement of Proposition \ref{key prop1}  holds for $r+1$.

This completes the proof of Proposition \ref{key prop1}  and Proposition \ref{key prop1'}.

\begin{rem}\label{rem key prop1} \begin{enumerate}
\item If we replace $\GL(V)^{\Delta^{(k)}}$ by $\GL(V)^{\Delta^{(r+1)}}$, we obtain the isomorphism $$\left(\omega_{(k),(k)}\right)_{\left(N_{(r)},\psi_{(r)}\right)}\cong\ind_{G^{\Delta^{(r+1),(r)}}\cdot H_{(k-r)}}^{\GL(V)^{\Delta^{(r+1)}}\times\GL(W_{(k)})}\left(\omega_{(k-r),(k-r-1)}\otimes\psi_{(r-1)}^{-1}\right),$$ which will be used in the next section. 
\item If we only consider the action of $\GL(W_{(k)})$ (and forget the action of $\GL(V)^{\Delta^{(k)}}$) in the proof of  Proposition \ref{key prop1}  and Proposition \ref{key prop1'}, the same argument yields the following isomorphism
$$\left(\omega_{(k),(k)}\right)_{\left(N_{(r)},\psi_{(r)}\right)}\cong\ind_{H_{(k-r)}}^{\GL(W_{(k)})}\left(\omega_{(k-r),(k-r-1)}\otimes\psi_{(r-1)}^{-1}\right).$$
\end{enumerate}
\end{rem}

\subsection{Proof of Theorem \ref{thm intro1-new}} Now we are ready to prove Theorem \ref{thm intro1-new}.  For $\pi\in\Irr(\GL(W_{(k)}))$ and $\sigma\in\Irr(\GL_n)$,
by Proposition \ref{key prop1} and Lemma \ref{lem frob}, we have
$$\begin{aligned}&\Hom_{\GL(V)^{\Delta^{(k)}}\cdot N_{(k-1)}}\left(\Theta(\pi),\sigma\otimes\psi_{(k-1)}\right)\\
\cong&\Hom_{\left(\GL(V)^{\Delta^{(k)}}\cdot N_{(k-1)}\right)\times\GL(W_{(k)})}\left(\Theta(\pi)\otimes\pi,\sigma\otimes\psi_{(k-1)}\otimes\pi\right)\\
\cong&\Hom_{\left(\GL(V)^{\Delta^{(k)}}\cdot N_{(k-1)}\right)\times\GL(W_{(k)})}\left(\omega_{(k),(k)},\sigma\otimes\psi_{(k-1)}\otimes\pi\right)\\
\cong&\Hom_{\GL(V)^{\Delta^{(k)}}\times\GL(W_{(k)})}\left(\left(\omega_{(k),(k)}\right)_{\left(N_{(k-1)},\psi_{(k-1)}\right)},\sigma\otimes\pi\right)\\
\cong&\Hom_{\GL(V)^{\Delta^{(k)}}\times\GL(W_{(k)})}\left(\ind_{G^{\Delta^{(k),(k-1)}}\cdot H_{(1)}}^{\GL(V)^{\Delta^{(k)}}\times\GL(W_{(k)})}\left(\BC\otimes\psi_{(k-2)}^{-1}\right),\sigma\otimes\pi\right)\\
\cong&\Hom_{G^{\Delta^{(k),(k-1)}}\cdot H_{(1)}}\left(\sigma^\vee\otimes\pi^\vee,\delta_{G^{\Delta^{(k),(k-1)}}\cdot H_{(1)}}^{1/2}\otimes\psi_{(k-2)}\right).
\end{aligned}$$
Note that under the isomorphism
$$\begin{tikzcd}[row sep=0.1cm]
	\GL(V)^{\Delta^{(k)}}\times\GL(W_{(k)}) & \GL(W)^{\Delta^{(k-1)}}\times\GL(W_{(k)}) \\
	\left(\Delta^{(k)}(g),h\right) & \left(\Delta^{(k-1)}(g),h\right),
	\arrow[from=1-1, to=1-2]
	\arrow[maps to, from=2-1, to=2-2]
\end{tikzcd}$$
the subgroup $G^{\Delta^{(k),(k-1)}}$ is sent to the subgroup $$\left\{(\Delta^{(k-1)}(g),\Delta^{(k-1)}(g))\mid g\in\GL(W)\right\}$$ of $\GL(W)^{\Delta^{(k-1)}}\times\GL(W_{(k)})$. Therefore we obtain
$$\Hom_{\GL(V)^{\Delta^{(k)}}\cdot N_{(k-1)}}\left(\Theta(\pi),\sigma\otimes\psi_{(k-1)}\right)\cong\Hom_{\GL(W)^{\Delta^{(k-1)}}\cdot H_{(1)}}\left(\pi^\vee,\delta_{\GL(W)^{\Delta^{(k-1)}}\cdot H_{(1)}}^{1/2}\otimes\sigma\otimes\psi_{(k-2)}\right).$$
This completes the proof of Theorem \ref{thm intro1-new}. 

\section{Proof of Theorem \ref{thm intro2}}

We first introduce more notation to rephrase Theorem \ref{thm intro2}.

For each $1\leq r\leq k-1$, we denote by $P(V_{(r)})$ (resp. $P(W_{(r)})$) the standard parabolic subgroup of $\GL(V_{(r)})$ (resp. $\GL(W_{(r)})$) attached to the partition $(n,\cdots,n)$ of $rn$, and by $M(V_{(r)})$ (resp. $M(W_{(r)})$) its standard Levi subgroup.  For $r=0$, let $P(V_{(0)})$ and $P(W_{(0)})$ be the trivial subgroups.  Denote
$$
P_{(r)}:=P(V_{(r)})N_{(k-r-1)}\quad\textrm{and}\quad Q_{(r)}:=P(W_{(r)})U_{(k-r-1)}.$$ 
In particular, $$P_{(0)}=N_{(k-1)}\quad\textrm{and}\quad P_{(1)}=G_{(1)}.$$
Under the identification of $\GL(V_{(r)})$ with $\GL(W_{(r)})$, we identify the corresponding subgroups with each other in the obvious way.

For  a character  $\xi_{(r)}=\otimes_{j=k-r+1}^k\xi_j$ of $M(V_{(r)})=\prod_{j=k-r+1}^k\GL(V_j)$, and an integer $r'\geq1$, we denote by
$$\xi_{(r)}\otimes \underbrace{1\otimes1\cdots\otimes1}_{r'\ \textrm{times}}$$ the character of $M(V_{(r+r')})$ defined in the following way:
\begin{itemize}
\item its restriction to $\GL(V_{k-r-r'+1})\times\GL(V_{k-r-r'+2})\times\cdots\times\GL(V_{k-r'})$ is the character $\xi_{(r)}$,
\item its restriction to $\GL(V_{k-r'+1})\times\GL(V_{k-r'+2})\times\cdots\times\GL(W_k)$ is the trivial character.
\end{itemize}

Recall that we have defined the subgroup  $G_{(r)}=\GL(V_{(r)})N_{(k-r-1)}$. Also
recall that $$\delta_{G_{(r)}}(g)=|\det\nolimits_{V_{(r)}}|^{(1+r-k)n},\quad g\in\GL(V_{(r)}).$$
We define two specific characters
$$\nu_{(r)}:=\delta_{G_{(r)}}^{1/2}\cdot(\delta^{1/2}_{G_{(r-1)}}\otimes 1)\cdot(\delta^{1/2}_{G_{(r-2)}}\otimes1\otimes1)\cdots(\delta^{1/2}_{G_{(1)}}\otimes \underbrace{1\otimes1\cdots\otimes1}_{r-1\ \textrm{times}})$$ and
$$\mu_{(r)}:=\delta_{P_{(r)}}^{1/2}\cdot(\delta^{1/2}_{G_{(r-1)}}\otimes 1)\cdot(\delta^{1/2}_{G_{(r-2)}}\otimes1\otimes1)\cdots(\delta^{1/2}_{G_{(1)}}\otimes \underbrace{1\otimes1\cdots\otimes1}_{r-1\ \textrm{times}})$$ of $M(V_{(r)})$, where $\delta_{G_{(i)}}$ denotes its restriction to $M(V_{(i)})$.
Note that
\begin{equation}\label{equ char-relation1}\delta_{P_{(r)}}=\delta_{P(V_{(r)})}\cdot\delta_{G_{(r)}}.\end{equation}
Hence 
\begin{equation}\label{equ char-relation2}\mu_{(r)}=\delta_{P(V_{(r)})}^{1/2}\cdot\nu_{(r)}.\end{equation}
Let
$$\chi_r:=\left(\delta^{1/2}_{\GL(W)^{\Delta^{(k-r)}}\cdot H_{(r)}}\right)\mid_{\GL(W)^{\Delta^{(k-r)}}},$$ which is a character of  $\GL(W)^{\Delta^{(k-r)}}$, and also viewed as a character of $\GL(W)$. Explicitly, we have
$$\chi_r(g)=|\det\nolimits_W(g)|^{(k-r-1)rn/2},\quad g\in\GL(W).$$ Put
$$\chi_{(k-1)}:=\prod_{r=1}^{k-1}\chi_r.$$
In terms of the above notation, Theorem \ref{thm intro2} can be rephrased as the second statement of the following theorem. 

\begin{thm}\label{thm2} Let $\pi\in\Irr(\GL(W_{(k)}))$ and $\sigma\in\Irr(\GL_n)$. For $0\leq r\leq k-2$, we have
$$\begin{aligned}&\Hom_{\GL(V)^{\Delta^{(k-r)}}\cdot P_{(r)}}\left(\Theta(\pi),\sigma\otimes\mu_{(r)}\otimes\psi_{(k-r-1)}\right) \\\cong&\Hom_{\GL(W)^{\Delta^{(k-r-1)}}\cdot Q_{(r+1)}}\left(\pi^\vee,\left(\chi_{r+1}\otimes\sigma\right)\otimes\mu_{(r+1)}\otimes\psi_{(k-r-2)}\right).\end{aligned}$$
Moreover, if $\Theta(\pi)\cong\pi^\vee$, then
$$\Hom_{\GL(W)^{\Delta^{(k)}}\cdot U_{(k-1)}}\left(\pi,\sigma\otimes\psi_{(k-1)}\right)\cong\Hom_{\GL(W)^{\Delta^{(1)}}\cdot Q_{(k-1)}}\left(\pi,\left(\chi_{(k-1)}\otimes\sigma\right)\otimes\mu_{(k-1)}\right).$$ 
\end{thm}
We will prove Theorem \ref{thm2} at the end of this section.

\subsection{Jacquet modules of the Weil representation}
Let us first recall Kudla's filtration of the normalized Jacquet module of the Weil representation. In this subsection, we denote $H=\GL_m$ and $G=\GL_n$. If a subgroup of $H$ (resp. $G$) is isomorphic to $\GL_k$, we denote it by $H_k$ (resp. $G_k$).  We denote by $\omega_{m,n}$ the Weil representation of the dual pair $(H,G)$. For $0\leq k\leq n$, let $P_{k,n-k}$ be the standard parabolic subgroup of $G$ with Levi component $M_{k,n-k}=G_k\times G_{n-k}$. Let $R_{P_{k,n-k}}(\omega_{m,n})$ be the normalized Jacquet module of $\omega_{m,n}$. M\'{\i}nguez \cite[Proposition 3.2]{min} proved that $R_{P_{k,n-k}}(\omega_{m,n})$ possesses a filtration, called  Kudla's filtration. We record this filtration in the form restated in \cite[Proposition 3.4]{chen}.

\begin{prop}\label{prop jacquet mod}
As an $M_{k,n-k}\times H$-module, $R_{P_{k,n-k}}(\omega_{m,n})$ has a filtration
$$0=J_{k+1}\subset J_k\subset J_{k-1}\subset\cdots\subset J_0=R_{P_{k,n-k}}(\omega_{m,n}),$$ and for each $0\leq i\leq k$, the successive quotient $\lambda_i:=J_i/J_{i+1}$ is given by
$$\lambda_i\cong\Ind_{P_{k-i,i}\times G_{n-k}\times Q_{i,m-i}}^{M_k\times H}\left(\mu_{i}\otimes\sigma_i\otimes\omega_{m-i,n-k}\right).$$ Here:
\begin{itemize}
\item $P_{k-i,i}$ is the standard parabolic subgroup of $G_k$ with Levi component $G_{k-i}\times G_i$;
\item $Q_{i,m-i}$ is the standard parabolic subgroup of $H$ with Levi component $H_{i}\times H_{m-i}$;
\item $\sigma_i$ is the natural geometric action of $G_i\times H_i$ on $\CS(\GL_i)$;
\item $\omega_{m-i,n-k}$ is the Weil representation of $H_{m-i}\times G_{n-k}$;
\item $\mu_{i}$ is a character of the standard Levi subgroup of $P_{k-i,i}\times G_{n-k}\times Q_{i,m-i}$, given by
$$\mu_{i}=\begin{cases} 
|\det|^{(m-n+k-i)/{2}} &\textrm{on}\ G_{k-i}\ ;\\
|\det|^{(m-n+2k-i)/{2}} &\textrm{on}\ G_{i}\ ; \\
|\det|^{(k-i)/{2}} &\textrm{on}\ G_{n-k}\ ; \\
|\det|^{(n-m-2k+i)/{2}} &\textrm{on}\ H_{i}\ ; \\
|\det|^{(i-k)/{2}} &\textrm{on}\ H_{m-i}.
 \end{cases}$$
\end{itemize}
\end{prop}

\subsection{Coinvariants of the Weil representation attached to parabolic subgroups} Recall that $\omega_{(r+1),(r)}$ is the Weil representation of $\GL(W_{(r+1)})\times\GL(V_{(r)})$. For $\chi\in\Irr(\GL(V_{(r)}))$, we denote by $\Theta_{(r+1),(r)}(\chi)$ its big theta lift to $\GL(W_{(r+1)})$.
Our proof of Theorem \ref{thm2} relies on the following proposition.

\begin{prop}\label{key prop2}
Let $\xi_{(r)}$ be a character of $M(V_{(r)})$ of the form $\otimes_{j=k-r+1}^k|\det\nolimits_V|^{\alpha_j}$, and set $$\xi'_{(r)}:=\delta^{1/2}_{P(V_{(r)})}\xi_{(r)}.$$ Suppose that
\begin{equation}\label{equ char-cond}\alpha_j\leq0,\quad\forall k-r+1\leq j\leq k.\end{equation} Then we have
$$\left(\omega_{(r+1),(r)}\right)_{\left(P(V_{(r)}),\xi'_{(r)}\right)}\cong\Ind^{\GL(W_{(r+1)})}_{P(W_{(r+1)})}\left(\xi^{-1}_{(r)}\otimes1\right).$$
\end{prop}
\begin{proof}
We prove by induction on $r$.  First, let us consider the case $r=1$. Note that $P(V_{(1)})=G_{(1)}$ and $\xi'_{(1)}=\xi_{(1)}$.
By \cite[Theorem 1.4]{chen}, we have
$$\begin{aligned}\left(\omega_{(2),(1)}\right)_{\left(P(V_{(1)}),\xi'_{(1)}\right)}=\Theta_{(2),(1)}\left(\xi_{(1)}\right)
\cong\Ind^{\GL(W_{(2)})}_{P(W_{(2)})}\left(\xi_{(1)}^{-1}\otimes 1\right).
\end{aligned}$$ 

Now assume that this proposition holds for $\left(\omega_{(r),(r-1)}\right)_{\left(P(V_{(r-1)}),\xi'_{(r-1)}\right)}$ for any character 
$\xi_{(r-1)}$ that satisfies the condition \eqref{equ char-cond}.
For $\left(\omega_{(r+1),(r)}\right)_{\left(P(V_{(r)}),\xi'_{(r)}\right)}$, we see that
$$\left(\omega_{(r+1),(r)}\right)_{\left(P(V_{(r)}),\xi'_{(r)}\right)}=\left(\left(\left(\omega_{(r+1),(r)}\right)_{N_{k-r+1}} \right)_{\left(\GL(V_{k-r+1}),\xi'_{(r)}\right)}\right)_{\left(P(V_{(r-1)}),\xi'_{(r)}\right)}.$$
Let $P(V_{k-r+1,(r-1)})$ denote the standard parabolic subgroup of $\GL(V_{(r)})$ whose standard Levi subgroup is $\GL(V_{k-r+1})\times\GL(V_{(r-1)})$. Then $N_{k-r+1}$ is the unipotent radical of $P(V_{k-r+1,(r-1)})$.
By Proposition \ref{prop jacquet mod}, the unnormalized Jacquet module $\left(\omega_{(r+1),(r)}\right)_{N_{k-r+1}}$ has a filtration whose successive quotients are $\delta^{1/2}_{P(V_{k-r+1,(r-1)})}\lambda_i$, where $0\leq i\leq n$, and
$$\lambda_i\cong\Ind_{P_{n-i,i}\times \GL(V_{(r-1)})\times Q_{i,(r+1)n-i}}^{\GL(V_{k-r+1})\times\GL(V_{(r-1)})\times\GL(W_{(r+1)})}\left(\mu_i\otimes\sigma_i\otimes\omega_{(r+1)n-i,(r-1)n}\right).$$ 
Note that
$$\delta_{P(V_{k-r+1,(r-1)})}|_{\GL(V_{k-r+1})}=\delta_{P(V_{(r)})}|_{\GL(V_{k-r+1})}.$$
Hence
$$\left(\delta^{1/2}_{P(V_{k-r+1,(r-1)})}\lambda_i \right)_{\left(\GL(V_{k-r+1}),\xi'_{(r)}\right)}=\left(\lambda_i \right)_{\left(\GL(V_{k-r+1}),\xi_{(r)}\right)}.
$$

Let us prove the following claim: \begin{equation}\label{equ claim-lambda}\left(\lambda_i \right)_{\left(\GL(V_{k-r+1}),\xi_{(r)}\right)}=0,\quad\forall 0\leq i\leq n-1.\end{equation} 
Since $\mu_{i}|_{\GL_{n-i}}=|\det|^{(2n-i)/2}$, we see that
$$\lambda_i\cong\Ind_{P_{n-i,i}}^{\GL(V_{k-r+1})}\left(|\det|^{(2n-i)/2}\otimes\tau\right)$$ as a representation of $\GL(V_{k-r+1})$, where $\tau$ is some smooth representation of $\GL_i$.
By Lemma \ref{lem bernstein}, we have
\begin{equation}\label{equ sec-adj1}\begin{aligned}\Hom_{\GL(V_{k-r+1})}\left(\lambda_i, \xi_{(r)}\right)
&\cong\Hom_{\GL_{n-i}\times\GL_i}\left(|\det|^{(2n-i)/2}\otimes\tau, R_{\overline{P}_{n-i,i}}\left(\xi_{(r)}\right)\right)\\
&\cong\Hom_{\GL_{n-i}\times\GL_i}\left(|\det|^{(2n-i)/2}\otimes\tau, \delta_{P_{n-i,i}}^{1/2}\xi_{(r)}\right).
\end{aligned}\end{equation}
The condition \eqref{equ char-cond} on $\xi_{(r)}$ implies that $$\xi_{(r)}|_{\GL_{n-i}}=|\det|^\alpha$$ for some $\alpha\leq0$.
If the quotient space in \eqref{equ claim-lambda} is nonzero,  the Hom-space in \eqref{equ sec-adj1} must be nonzero, forcing 
\begin{equation}\label{equ det}\begin{aligned}|\det|^{(2n-i)/2}=\left(\delta_{P_{n-i,i}}^{1/2}\xi_{(r)}\right)|_{\GL_{n-i}}
=|\det|^{i/2}\cdot|\det|^\alpha.
\end{aligned}\end{equation}
Since $\alpha\leq0$ and $i\leq n$, the equality \eqref{equ det} forces  $i=n$. Hence the claim \eqref{equ claim-lambda} holds.

Note that
$$\left(\delta^{-1}_{P(V_{k-r+1,(r-1)})}\delta_{P(V_{(r)})}\right)|_{P(V_{(r-1)})}=\delta_{P(V_{(r-1)})}.$$
Therefore the claim \eqref{equ claim-lambda} and the description of $\lambda_n$ (see Proposition \ref{prop jacquet mod}) imply that 
\begin{equation}\label{equ general omega_N}\begin{aligned}&\left(\left(\left(\omega_{(r+1),(r)}\right)_{N_{k-r+1}} \right)_{\left(\GL(V_{k-r+1}),\xi'_{(r)}\right)}\right)_{\left(P(V_{(r-1)}),\xi'_{(r)}\right)}\\
\cong&\left((\lambda_n)_{\left(\GL(V_{k-r+1}), \xi_{(r)}\right)}\right)_{\left(P(V_{(r-1)}),\delta^{1/2}_{P(V_{(r-1)})}\xi_{(r)}\right)}\\
\cong&\Ind_{P(W_{k-r,(r)})}^{\GL(W_{(r+1)})}\left(\begin{array}{c}\left(|\det\nolimits_{V_{k-r+1}}|^n\otimes\CS(\GL_n)\otimes|\det\nolimits_{W_{k-r}}|^{-n}\right)_{\left(\GL(V_{k-r+1}),\xi_{(r)}\right)}\\ \otimes\left(\omega_{(r),(r-1)}\right)_{\left(P(V_{(r-1)}),\delta^{1/2}_{P(V_{(r-1)})}\xi_{(r)}\right)}\end{array}\right).\end{aligned}
\end{equation}
It is easy to see that
\begin{equation}\label{equ lambda1}\begin{aligned}\left(|\det\nolimits_{V_{k-r+1}}|^n\otimes\CS(\GL_n)\otimes|\det\nolimits_{W_{k-r}}|^{-n}\right)_{\left(\GL(V_{k-r+1}),\xi_{(r)}\right)}
=\left(\xi^{-1}_{(r)}\otimes1\right)|_{\GL(W_{k-r})}.
\end{aligned}\end{equation}
On the other hand, the character $\xi_{(r)}|_{M(V_{(r-1)})}$ satisfies the condition \eqref{equ char-cond} in the inductive hypothesis for the case $r-1$.
Hence we have
\begin{equation}\label{equ lambda2}\begin{aligned}\left(\omega_{(r),(r-1)}\right)_{\left(P(V_{(r-1)}),\delta^{1/2}_{P(V_{(r-1)})}\xi_{(r)}\right)}
\cong\Ind^{\GL(W_{(r)})}_{P(W_{(r)})}\left(\left(\xi^{-1}_{(r)}\otimes1\right)|_{P(W_{(r)})}\right).
\end{aligned}\end{equation}
Combining \eqref{equ general omega_N}, \eqref{equ lambda1} and \eqref{equ lambda1}, we obtain
$$\begin{aligned}
(\omega_{(r+1),(r)})_{(P(V_{(r)}),\xi'_{(r)})}
\cong&\Ind_{P(W_{k-r,(r)})}^{\GL(W_{(r+1)})}\left(\left(\xi^{-1}_{(r)}\otimes1\right)|_{\GL(W_{k-r})}\otimes\Ind^{\GL(W_{(r)})}_{P(W_{(r)})}\left(\left(\xi^{-1}_{(r)}\otimes1\right)|_{P(W_{(r)})}\right) \right)\\
\cong&\Ind^{\GL(W_{(r+1)})}_{P(W_{(r+1)})}(\xi^{-1}_{(r)}\otimes1).
\end{aligned}$$
This completes the proof of the proposition.
\end{proof}

\subsection{Proof of Theorem \ref{thm2}}\label{subsec pf of thm2}  Now we are ready to prove Theorem \ref{thm2}. 
By Proposition  \ref{key prop1} and Remark \ref{rem key prop1}, we have
$$\begin{aligned}&\Hom_{\GL(V)^{\Delta^{(k-r)}}\cdot P_{(r)}}\left(\Theta(\pi),\sigma\otimes\mu_{(r)}\otimes\psi_{(k-r-1)}\right)\\ 
\cong&\Hom_{\left(\GL(V)^{\Delta^{(k-r)}}\cdot P_{(r)}\right)\times\GL(W_{(k)})}\left(\omega_{(k),(k)},\sigma\otimes\mu_{(r)}\otimes\psi_{(k-r-1)}\otimes\pi\right)\\
\cong&\Hom_{\GL(V)^{\Delta^{(k-r)}}\times P(V_{(r)})\times\GL(W_{(k)})}\left(\left(\omega_{(k),(k)}\right)_{\left(N_{(k-r-1)},\psi_{(k-r-1)}\right)},\sigma\otimes\mu_{(r)}\otimes\pi\right)\\
\cong&\Hom_{\GL(V)^{\Delta^{(k-r)}}\times P(V_{(r)})\times\GL(W_{(k)})}\left(\ind_{G^{\Delta^{(k-r),(k-r-1)}}\cdot H_{(r+1)}}^{\GL(V)^{\Delta^{(k-r)}}\times\GL(W_{(k)})}\left(\omega_{(r+1),(r)}\otimes\psi_{(k-r-2)}^{-1}\right),\sigma\otimes\mu_{(r)}\otimes\pi\right)\\
\cong&\Hom_{\GL(V)^{\Delta^{(k-r)}}\times\GL(W_{(k)})}\left(\ind_{G^{\Delta^{(k-r),(k-r-1)}}\cdot H_{(r+1)}}^{\GL(V)^{\Delta^{(k-r)}}\times\GL(W_{(k)})}\left(\left(\omega_{(r+1),(r)}\right)_{\left(P(V_{(r)}),\mu_{(r)}\right)}\otimes\psi_{(k-r-2)}^{-1}\right),\sigma\otimes\pi\right).
\end{aligned}$$ 
Note that $\mu_{(r)}$ satisfies the condition \eqref{equ char-cond}. Thus Proposition \ref{key prop2} yields
$$\left(\omega_{(r+1),(r)}\right)_{\left(P(V_{(r)}),\mu_{(r)}\right)}\cong\Ind^{\GL(W_{(r+1)})}_{P(W_{(r+1)})}\left(\nu^{-1}_{(r)}\otimes1\right).$$ Therefore
$$\begin{aligned}&\ind_{G^{\Delta^{(k-r),(k-r-1)}}\cdot H_{(r+1)}}^{\GL(V)^{\Delta^{(k-r)}}\times\GL(W_{(k)})}\left(\left(\omega_{(r+1),(r)}\right)_{\left(P(V_{(r)}),\mu_{(r)}\right)}\otimes\psi_{(k-r-2)}^{-1}\right)\\
\cong&\ind_{G^{\Delta^{(k-r),(k-r-1)}}\cdot H_{(r+1)}}^{\GL(V)^{\Delta^{(k-r)}})^\Delta\times\GL(W_{(k)})}\left(\Ind^{\GL(W_{(r+1)})}_{P(W_{(r+1)})}\left(\nu^{-1}_{(r)}\otimes1\right)\otimes\psi_{(k-r-2)}^{-1}\right)\\
\cong&\ind_{G^{\Delta^{(k-r),(k-r-1)}}\cdot Q_{(r+1)}}^{\GL(V)^{\Delta^{(k-r)}}\times\GL(W_{(k)})}\left(\left(\nu^{-1}_{(r)}\otimes1\right)\otimes\psi_{(k-r-2)}^{-1}\right).
\end{aligned}$$
Thus, as the proof of Theorem \ref{thm intro1}, we get
$$\begin{aligned}&\Hom_{\GL(V)^{\Delta^{(k-r)}}\cdot P_{(r)}}\left(\Theta(\pi),\sigma\otimes\mu_{(r)}\otimes\psi_{(k-r-1)}\right)\\
\cong& \Hom_{G^{\Delta^{(k-r),(k-r-1)}}\cdot Q_{(r+1)}}\left(\sigma^\vee\otimes\pi^\vee,\delta^{1/2}_{G^{\Delta^{(k-r),(k-r-1)}}\cdot Q_{(r+1)}}\otimes\left(\nu_{(r)}\otimes1\right)\otimes\psi_{(k-r-2)}\right)\\
\cong& \Hom_{\GL(W)^{\Delta^{(k-r-1)}}\cdot Q_{(r+1)}}\left(\pi^\vee,\left(\chi_{r+1}\otimes\sigma\right)\otimes\left(\delta^{1/2}_{Q_{(r+1)}}\cdot\left(\nu_{(r)}\otimes1\right)\right)\otimes\psi_{(k-r-2)}\right)\\
=& \Hom_{\GL(W)^{\Delta^{(k-r-1)}}\cdot Q_{(r+1)}}\left(\pi^\vee,\left(\chi_{r+1}\otimes\sigma\right)\otimes\mu_{(r+1)}\otimes\psi_{(k-r-2)}\right),
\end{aligned}$$
which proves the first part of the theorem.

If $\Theta(\pi)\cong\pi^\vee$, then $\Theta(\pi^\vee)\cong\pi$. Applying the Theorem \ref{thm intro1-new} as the initial step ($r=0$) and the above isomorphism repeatedly (from $r=1$ to $r=k-2$), 
we deduce 
the second statement of the theorem.

\s{\small Chong Zhang\\
School of Mathematics, Nanjing University,\\
Nanjing 210093, Jiangsu, P. R. China.\\
E-mail address: \texttt{zhangchong@nju.edu.cn}}

\end{document}